\documentclass[leqno]{amsart}
\usepackage{amsmath,amsfonts,amsthm,amssymb,indentfirst,epic,url,centernot,graphics}

\newcommand{\strt}[1][1.7]{\vrule width0pt height0pt depth#1pt}

\newtheorem{theorem}{Theorem}
\newtheorem{lemma}[theorem]{Lemma}
\newtheorem{corollary}[theorem]{Corollary}
\newtheorem{proposition}[theorem]{Proposition}
\newtheorem{definition}[theorem]{Definition}

\newtheorem{example}[theorem]{Example}

\newcommand{\Z}{\mathbb{Z}}

\newcommand{\card}{\mathrm{card}}
\newcommand{\T}{\mathcal{T}}

\newcommand{\lang}{\begin{picture}(5,7)
\put(1.1,2.5){\rotatebox{45}{\line(1,0){6.0}}}
\put(1.1,2.5){\rotatebox{315}{\line(1,0){6.0}}}
\end{picture}}
\newcommand{\rang}{\begin{picture}(5,7)
\put(.1,2.5){\rotatebox{135}{\line(1,0){6.0}}}
\put(.1,2.5){\rotatebox{225}{\line(1,0){6.0}}}
\end{picture}}

\begin{document}

\begin{center}
\texttt{Comments, corrections, and related references welcomed,
as always!}\\[.5em]
{\TeX}ed \today
\vspace{2em}
\end{center}

\title{On group topologies determined by families of sets}
\thanks{This preprint is readable online at
\url{http://math.berkeley.edu/~gbergman/papers/}.
}


\author{George M. Bergman}
\address{University of California\\
Berkeley, CA 94720-3840, USA}
\email{gbergman@math.berkeley.edu}

\begin{abstract}
Let $G$ be an abelian group, and $F$ a downward directed family
of subsets of $G.$
In \cite{P+Z}, I.\,Protasov and E.\,Zelenyuk describe the finest
group topology $\T$ on $G$ under which $F$ converges to $0;$
in particular, their description yields a criterion for $\T$ to be
Hausdorff.
They then show that if $F$ is the filter of cofinite subsets
of a countable subset $X\subseteq G$
(the Fr\'{e}chet filter on $X),$ then
there is a simpler criterion: $\T$
is Hausdorff if and only if for every $g\in G-\{0\}$ and
positive integer $n,$ there is an $S\in F$ such
that $g$ does not lie in the $\!n\!$-fold sum $n\,(S\cup\{0\}\cup -S).$

In this note, their proof is adapted to a larger class of families $F.$
In particular, if $X$ is any infinite subset of $G,$ $\kappa$
any regular infinite cardinal $\leq\card(X),$ and $F$ the set of
complements in $X$ of subsets of cardinality $<\kappa,$ then the above
criterion holds.

We then give some negative examples, including a countable
downward directed set $F$ of subsets of $\Z$ not
of the above sort, which satisfies
the ``$g\notin n\,(S\cup\{0\}\cup -S)$'' condition, but does not
induce a Hausdorff topology.

We end with a version of our main result for noncommutative $G.$
\end{abstract}
\maketitle

\section{Introduction}\label{S.intro}

Let $G$ be a group, let $F$ be a set of subsets of $G$ which is
downward directed, i.e., such that whenever $S_1,\,S_2\in F,$
there is an $S_3\in F$ which is contained in $S_1\cap S_2,$
and let $\T$ be a {\em group topology} on $G;$
that is, a (not necessarily Hausdorff) topology under which
the group multiplication and inverse operation are continuous.
We say that $F$ converges to an element $x\in G$ under $\T$
if every $\!\T\!$-neighborhood of $x$ contains a member of~$F.$

Given $G$ and $F,$ it is not hard to show that there
will exist a finest group topology $\T_F$
on $G$ under which $F$ converges to the identity element of $G.$
The explicit description of $\T_F$ is simpler and easier to study
for abelian $G$ than for general $G,$ so we shall assume
for most of this note that
\begin{equation}\begin{minipage}[c]{35pc}\label{d.abelian}
$G$ is abelian, with operations written additively.
\end{minipage}\end{equation}

To describe the topology $\T_F,$ let us set up some notation.
For any subset $S\subseteq G,$ let
\begin{equation}\begin{minipage}[c]{35pc}\label{d.S^*}
$S^*\ =\ S\cup\{0\}\cup -S.$
\end{minipage}\end{equation}
For any sequence of subsets $S_0,\,S_1,\,\dots\subseteq G$
indexed by the set $\omega$ of natural numbers, let
\begin{equation}\begin{minipage}[c]{35pc}\label{d.U}
$U(S_0,S_1,\dots)\ =\ \bigcup_{n\in\omega}\,\sum_{i<n}\,S_i^*
\ =\ \{x_0+\dots+x_{n-1}\mid n\in\omega,\ x_i\in S_i^*\}.$
\end{minipage}\end{equation}
(The $n=0$ term of the above union, i.e., the sum of the vacuous
sequence of sets, is understood to be $\{0\}.)$
Then one has
\begin{equation}\begin{minipage}[c]{35pc}\label{d.U-basis}
\cite[Lemma 2.1.1]{P+Z}\ \ The sets $U(S_0,S_1,\dots),$ as
$(S_i)_{i\in\omega}$
runs over all sequences of elements of $F,$ form a basis of open
neighborhoods of $0$ under $\T_F,$ the finest group topology
on $G$ under which $F$ converges to $0.$
\end{minipage}\end{equation}

Thus, as noted in \cite[Theorem~2.1.3]{P+Z}, the topology $\T_F$ is
Hausdorff (equivalently, there exists a Hausdorff group topology
under which $F$ converges to $0)$ if and only if
\begin{equation}\begin{minipage}[c]{35pc}\label{d.capU=0}
$\bigcap_{S_0,\,S_1,\,\dots\,\in F}\,U(S_0,\,S_1,\,\dots\,)\ =\ \{0\}.$
\end{minipage}\end{equation}

(Our formulations of these statements are different from those in
\cite{P+Z} because there, group topologies are by definition
Hausdorff.
Though Hausdorff topologies are the interesting ones, we find it
convenient, for making statements like~(\ref{d.U-basis}),
to allow non-Hausdorff topologies as well.
Incidentally, a topological group is Hausdorff if and
only if it is $\mathrm{T}_0$
\cite[p.32, Proposition~4 and preceding Exercise]{PJH}.)

From the fact that~(\ref{d.capU=0}) is necessary and sufficient
for $\T_F$ to be Hausdorff, we get a weaker condition which is
necessary.

\begin{corollary}\label{C.cupcap=0}
A \emph{necessary} condition for the topology $\T_F$ to be Hausdorff is
\begin{equation}\begin{minipage}[c]{35pc}\label{d.cupcap=0}
$\bigcup_{n>0}\,\bigcap_{S\in F_{\strt}}\,n\,S^*\ =\ \{0\}.$
In other words, for every $g\in G-\{0\}$ and every $n>0,$
there exists $S\in F$ with $g\notin n\,S^*.$
\end{minipage}\end{equation}
\end{corollary}

\begin{proof}
Assuming~(\ref{d.capU=0}), consider any $g\in G-\{0\}$ and any $n>0.$
By~(\ref{d.capU=0})
we can choose $S_0, S_1,\dots$ such that $g\notin U(S_0,\,S_1,\,\dots).$
Hence $g$ fails to lie in the smaller set $S_0^*+\dots+S_{n-1}^*.$
Letting $S$ be a common lower bound for $S_0,\dots,S_{n-1}$ in
the downward directed set $F,$ $g$ therefore
fails to lie in $n\,S^*,$ as required.
\end{proof}

As illustrated by the notation $G-\{0\}$ above, in
this note a ``$-$'' between sets indicates {\em relative complement};
thus, $X-Y$ never denotes $X+(-Y).$

In \S\ref{S.eg}, we shall see by example
that~(\ref{d.cupcap=0}) is not in general \emph{sufficient}
for $\T_F$ to be Hausdorff.
However, Protasov and Zelenyuk \cite[Theorem 2.1.4]{P+Z}
show that it \emph{is} sufficient if $F$ is the
filter of cofinite subsets of a countable subset
$X=\{x_0,\,x_1,\,\dots\,\}\subseteq G$
(often called the Fr\'{e}chet filter on $X);$
in other words, if $\T_F$
is the finest group topology on $G$ making $\lim_{i\to\infty} x_i=0.$
Generalizing their argument, we shall obtain below
the same result for a wider class of $F.$
In \S\ref{S.nonab} we shall extend this result to nonabelian $G.$

\section{Co-$\!\kappa\!$ filters, and a peculiar condition that they satisfy}\label{S.peculiar}

Here is our generalization of
the class of filters considered in~\cite{P+Z}.

\begin{definition}[{\cite[Example~II.2.5]{E+M}}]\label{D.co-kappa}
Let $X$ be an infinite set and $\kappa$ an infinite cardinal
$\leq\card(X).$
Then by the \emph{co-$\!\kappa\!$ filter} on $X$
we shall mean the \textup{(}downward directed\textup{)}
set of complements in $X$ of subsets with cardinality $<\kappa.$
For $\kappa=\aleph_0,$ this will be called the \emph{cofinite}
filter on $X.$
\end{definition}

(Remark:  The cofinite filter on an infinite set $X$
is often called the Fr\'{e}chet filter on $X.$
In some places, the co-$\!\card(X)\!$ filter on $X$ has been called
the ``generalized Fr\'{e}chet filter'';
in \cite[p.197]{UA} the term ``Fr\'{e}chet filter''
is used, instead, for the latter construction.)

To state the property of these filters that we
will use, we make the following definition.
It has the same form as the
definition of convergence of a family of points
under a group topology on $G,$ but with the system of
neighborhoods of $0$ replaced by a more general family.

\begin{definition}\label{D.strong}
Suppose $F$ is any downward directed family of subsets of the abelian
group $G,$ and $(x_i)_{i\in I}$ a family of elements of $G$ indexed by a
downward directed partially ordered set $I.$
We shall say that $(x_i)_{i\in I}$ ``converges strongly'' to
an element $x\in G$ with respect to $F$ if for every $S\in F,$
there exists $i\in I$ such that for all $j\leq i,$ $x_j-x\in S^*.$
\end{definition}

(Since $F$ is not assumed to be a neighborhood basis of a
group topology, this is not a very natural condition.
I use the modifier ``strongly'' because the condition is stronger than
convergence in the group topology \emph{determined}
by $F$ as in~(\ref{d.U-basis}).
Note, incidentally, that the way in which the ordering
on $I$ is used in Definition~\ref{D.strong} is the reverse of the usual.
This is not essential; it simply spares us introducing the opposite
of a natural ordering below.
Indeed, when an index set $I$ is
described as \emph{downward} rather than upward directed,
it is natural to adjust accordingly what one understands
convergence of an $\!I\!$-indexed family to mean.)

We can now state the condition around which our main result will center.

\begin{definition}\label{D.self-indulgent}
A downward directed family $F$ of subsets of the abelian group $G$
will be called \emph{self-indulgent} if for
every $S\in F,$ and every family $(x_T)_{T\in F'}$ of elements
of $S^*$ indexed by a downward cofinal
subset $F'\subseteq F,$ there exist an $x\in S^*,$
and a downward cofinal subset $F''\subseteq F',$ such that
$(x_T)_{T\in F''}$ converges strongly to $x$ with respect to $F.$
\end{definition}

A strange feature of this condition (which motivates its name)
is that it involves the family $F$ in three unrelated ways:
First, $S$ is taken to be a member of $F;$ second, the
family of points $x_T\in S^*$ is indexed by a
subfamily of $F,$ and third, the convergence asked for is strong
convergence with respect to $F.$

\begin{lemma}\label{L.Fr-self-ind}
Let $X$ be any infinite subset of the abelian group $G,$
and $\kappa$ any regular infinite cardinal $\leq\card(X).$
Then the co-$\!\kappa\!$ filter $F$ on $X$ is self-indulgent
as a family of subsets of $G.$
\end{lemma}

\begin{proof}
Let $S\in F,$ and let $(x_T)_{T\in F'}$ be a family of elements
of $S^*$ indexed by a downward cofinal subset $F'\subseteq F.$
If there exists an $x\in S^*$ which occurs ``frequently'' as a value
of $x_T,$ in the sense that $\{T\in F'\mid x_T = x\}$ is downward
cofinal in $F',$ then for this $x,$ and $F''=\{T\in F'\mid x_T = x\},$
the condition of Definition~\ref{D.self-indulgent}
is trivially satisfied:
for $T\in F''$ we have $x-x_T=0,$ which belongs
to $R^*$ for all $R\in F.$

If there is no such ``frequently occurring'' value, then I
claim we can use $F''=F'$ and $x=0.$
Indeed, again writing $R\in F$ for the set called $S$ in
the definition of strong convergence (since we already
have a set we are calling $S),$ note that for every
such $R$ we have $\card(S^*- R^*) <\kappa;$
and for each $s\in S^*- R^*,$ the fact that $s$ does not
occur ``frequently'' among the $x_T$ tells us that we can
find $T_s\in F'$ such that no $x_T$ with $T\subseteq T_s$
and $T\in F'$ is equal to $s.$
If we let $T_0$ be the intersection of these $T_s$
over all $s\in S^*- R^*,$ then
by regularity of the cardinal $\kappa,$ we
have $T_0\in F,$ hence by downward cofinality of $F'$ in $F,$
the set $F'$ contains some $T_{R}\subseteq T_0.$
For all $T\subseteq T_{R}$ in $F'$ we have
$x_T-0\in R^*;$ so $(x_T)_{T\in F'}$
converges strongly to $0$ with respect to $F.$
\end{proof}

\section{Our main result}\label{S.main}

We shall now prove that for $F$ a self-indulgent family,
and $\T_F$ the topology it determines, we have
(\ref{d.capU=0})$\!\iff\!$(\ref{d.cupcap=0}).
Here (\ref{d.capU=0})$\!\implies\!$(\ref{d.cupcap=0})
is Corollary~\ref{C.cupcap=0}.
The plan of our proof of the converse will be to
show that, given $g\in G-\{0\}$ which we want to exclude from the
intersection in~(\ref{d.capU=0}), we can build up,
in a recursive manner, a sequence
$S_0,\,S_1,\,\dots$ with $g\notin U(S_0,\,S_1,\dots).$
The recursive step is given by the next lemma.
(The corresponding recursive step in the proof of
\cite[Theorem 2.1.4]{P+Z} uses an ``either/or'' argument at each
substep.
These were collapsed here into the single either/or argument
in the above proof that co-$\!\kappa\!$ filters on subsets of $G$
are self-indulgent.)

\begin{lemma}\label{L.one-more}
Let $F$ be a self-indulgent downward directed system of subsets of
$G$ satisfying~\textup{(\ref{d.cupcap=0})}.
Suppose $g\in G-\{0\},$ and that for some $n\geq 0,$
$S_0,\dots,S_{n-1}$ are members of $F$ such that
\begin{equation}\begin{minipage}[c]{35pc}\label{d.n-1}
$g\notin S_0^*+\dots+S_{n-1}^*.$
\end{minipage}\end{equation}
Then there exists $S_n\in F$ such that
\begin{equation}\begin{minipage}[c]{35pc}\label{d.n}
$g\notin S_0^*+\dots+S_{n-1}^*+S_n^*.$
\end{minipage}\end{equation}
\end{lemma}

\begin{proof}
Assume the contrary.
Then for each $T\in F,$ the fact that~(\ref{d.n}) does not hold
with $S_n=T$ shows that we may choose $n+1$ elements,
\begin{equation}\begin{minipage}[c]{35pc}\label{d.nth}
$g_{0,T}\in S_0^*,\ \ \dots\,,\ \ g_{n{-}1,T}\in S_{n-1}^*,\ \ g_{n,T}
\in T^*$
\end{minipage}\end{equation}
such that
\begin{equation}\begin{minipage}[c]{35pc}\label{d.=g}
$g\ =\ g_{0,T}+\dots+g_{n{-}1,T}+g_{n,T}.$
\end{minipage}\end{equation}

Assuming for the moment that $n>0,$ let us
focus on the first term on the right-hand side of~(\ref{d.=g}),
and apply the assumption that $F$ is self-indulgent to the family of
elements $g_{0,T}\in S_0^*,$ as $T$ ranges over $F.$
This tells us that we can find a
$g_0\in S_0^*$ and a downward cofinal subset $F_0\subseteq F$ such that
\begin{equation}\begin{minipage}[c]{35pc}\label{d.F_0}
$(g_{0,T})_{T\in F_0}$ converges strongly to $g_0$ with
respect to $F.$
\end{minipage}\end{equation}

If $n>1,$ we then go through the same process for
the values $g_{1,T}\in S_1^*,$ as
$T$ ranges over the above downward cofinal subset $F_0\subseteq F.$
By the self-indulgence of $F,$ we can find
a $g_1\in S_1^*$ and a downward cofinal
subset $F_1$ of $F_0,$ such that
\begin{equation}\begin{minipage}[c]{35pc}\label{d.F_1}
$(g_{1,T})_{T\in F_1}$ converges strongly to $g_1$ with
respect to $F.$
\end{minipage}\end{equation}

We continue this way, through the construction
of $g_{n-1}$ and $F_{n-1}.$
We do not have to do anything at the next step, but simply set
$F_n = F_{n-1}$ (or if $n=0,$ $F_n=F),$
and $g_n = 0,$ since the assumption $g_{n,T}\in T^*$
in~(\ref{d.nth}) says that the family $(g_{n,T})_{T\in F}$ already
converges strongly to $0,$ whence the same holds when we restrict
the index $T$ to the cofinal subset $F_{n-1}\subseteq F.$

Now since $g_i\in S_i^*$ for $i<n,$ while $g_n=0,$ we have
$g_0+\dots+g_n\in S_1^*+\dots+S_{n-1}^*,$ so by~(\ref{d.n-1}),
$g\neq g_0+\dots+g_n.$
Letting $g'=g-(g_0+\dots+g_n)\neq 0,$
condition~(\ref{d.=g}) becomes
\begin{equation}\begin{minipage}[c]{35pc}\label{d.=g'}
$g'\ =\ (g_{0,T}-g_0)+\dots+(g_{n,T}-g_n)$ for all $T\in F_n.$
\end{minipage}\end{equation}

We now apply our hypothesis that $F$ satisfies~(\ref{d.cupcap=0}).
Since $g'\neq 0,$ this says there is some $S\in F$ such that
\begin{equation}\begin{minipage}[c]{35pc}\label{d.g'notin}
$g'\notin(n+1)S^*.$
\end{minipage}\end{equation}

But since for each $i,$ the system $(g_{i,T}-g_i)_{T\in F_n}$ converges
strongly to $0,$ we can find $T\in F_n$ such that each
element $g_{i,T}-g_i$ $(0\leq i\leq n)$
lies in $S^*.$
Thus,~(\ref{d.=g'}) contradicts~(\ref{d.g'notin}),
and this contradiction completes the proof of the lemma.
\end{proof}

We deduce

\begin{theorem}[{cf.\ \cite[Theorem 2.1.4]{P+Z}}]\label{T.main}
Let $F$ be a self-indulgent downward directed
system of subsets of an abelian group $G.$
\textup{(}In particular, by Lemma~\ref{L.Fr-self-ind}, for
any infinite $X\subseteq G$ and any $\kappa\leq\card(X),$ such
an $F$ is given by the co-$\!\kappa\!$ filter on $X.)$
Then the finest group topology on $G$ under which $F$ converges
to $0$ is Hausdorff if and only
if $F$ satisfies~\textup{(\ref{d.cupcap=0})}.
\end{theorem}

\begin{proof}
By Corollary~\ref{C.cupcap=0},~(\ref{d.cupcap=0}) is
necessary for our topology to be Hausdorff.
Conversely, assuming~(\ref{d.cupcap=0}),
we can use Lemma~\ref{L.one-more} recursively to
build up, for any $g\in G-\{0\},$
a sequence $S_0,\,S_1,\,\dots$ of members of $F,$ starting with
the vacuous sequence, such that for all $n,$ $g\notin\sum_{i<n} S_i^*.$
Thus, $g\notin U(S_0,S_1,\dots),$ giving~(\ref{d.capU=0}), which
is equivalent to our topology being Hausdorff.
\end{proof}

One may ask whether allowing co-$\!\kappa\!$ filters with
$\kappa$ strictly less than $\card(X)$ provides any useful examples.
Such a filter only ``scratches the surface'' of $X,$ so it might
seem implausible that it could converge to $0$ in a group topology.
But in fact, if $G$ is the group $\Z^I$ for an
uncountable set $I,$ under the product
topology, and $X$ the set of elements of $G$ which have value $1$
at a single point, and $0$ everywhere else, then we see that the
cofinite (i.e., co-$\!\aleph_0\!)$ filter determined by $X$
does converge to $0$ in $G.$

\section{Some counterexamples}\label{S.eg}

Before giving the rather complicated example showing
that Theorem~\ref{T.main} fails if the assumption that $F$ is
self-indulgent is removed,
let us note a couple of easier cases of
things that go wrong in the absence of self-indulgence.

\begin{example}\label{E.not-one-more}
An abelian group $G$ with an element $g,$ a downward directed family
$F$ of subsets, and a sequence $S_0,\,\dots,\,S_{n-1}\in F$
satisfying~\textup{(\ref{d.n-1})}, which cannot,
as in Lemma~\ref{L.one-more}, be extended so as to
satisfy~\textup{(\ref{d.n})}.
\end{example}

\begin{proof}[Construction and proof]
Let $G$ be the additive group of the real line, $F$ the set of
neighborhoods $(-\varepsilon,\varepsilon)$ of $0$ $(\varepsilon>0),$
and $g = 1\in G.$
Then the $\!1\!$-term sequence given by $S_0 = (-1,1)$ satisfies
$g\notin S_0^*,$ but cannot be extended to a $\!2\!$-term
sequence with $g\notin S_0^* + S_1^*.$
\end{proof}

Indeed, whenever, as in the above example, $F$ consists of
neighborhoods of the identity in the topology we are constructing, then
the conclusion of Lemma~\ref{L.one-more}
implies that $S_0^*+\dots+S_{n-1}^*$ is closed in that topology.
So if, starting with a topological group $G,$ we
take a basis $F$ of neighborhoods of $0$ not
all of which are closed sets, the conclusion of that lemma must fail.

Getting closer to our main example, we give

\begin{example}\label{E.proper_vs_G}
An abelian group $G$ and a downward
directed family $F$ of subsets of $G$ such that
the union in~\textup{(\ref{d.cupcap=0})} is a proper subgroup of $G,$
but the intersection in~\textup{(\ref{d.capU=0})} is all of $G.$
\end{example}

\begin{proof}[Construction and proof]
Let $G$ be the countable direct product group $\prod_{n>0}\,\Z/n\Z,$
and for each positive integer $m,$ let $S(m)\subseteq G$ consist
of all elements whose first through $\!m\!$-th coordinates lie in
$\{-1,0,1\},$ the remaining coordinates being unrestricted.
Thus, $S(1)\supseteq S(2)\supseteq\dots,$ so $F=\{S(m)\}$ is
downward directed.
(These sets satisfy $S(m)^* = S(m),$ but I will write $S(m)^*$ below
when the conditions we want to verify refer to sets $S^*.)$

To show that the intersection in~(\ref{d.capU=0}) is all of $G,$
we will in fact show that for any $m_0,\,m_1,\,\dots\,,$ we have
\begin{equation}\begin{minipage}[c]{35pc}\label{d.=G}
$S(m_0)^*+\dots+S(m_{m_0})^*\ =\ G.$
\end{minipage}\end{equation}

Indeed, let $g\in G.$
To describe the $m_0$ summands comprising an expression
for $g$ as a member of $S(m_0)^*+\dots+S(m_{m_0})^*,$ we shall first
describe their coordinates in
$\Z/1\Z,\,\dots,\,\Z/m_0\Z,$ then their remaining coordinates.
We choose the former coordinates all to lie in $\{-1,0,1\},$ and
be chosen so that for each $i\leq m_0,$ the $\!i\!$-th coordinates
of these $m_0$ elements sum to the $\!i\!$-th coordinate of $g.$
This is possible because the relevant
coordinates of $g$ are members of groups $\Z/n\Z$ with $n\leq m_0.$

We then choose the coordinates after the $\!m_0\!$-th by taking
these coordinates of the summand in $S(m_0)$ to agree with
those of $g,$ and those in the other summands to be zero.
It is easy to see that the elements we have constructed
belong to the desired $S(i)^*$ and sum to $g.$

On the other hand, consider any $g$ in the union in~(\ref{d.cupcap=0}).
Say it lies in the member of that union indexed by $n\in\omega.$
Thus, for each $m,$ $g$ lies in $n\,S(m)^*;$ i.e., for
each $m,$ the first $m$ coordinates of $g$ are sums of $n$ terms
in $\{0,1,-1\};$ i.e., each is the residue of an integer of absolute
value $\leq n.$
Since this is so for every $m,$ \emph{every} coordinate
of $g$ is the residue of an integer of absolute value $\leq n.$
Elements having this property for some $n$ clearly form a proper
subgroup of $G.$
\end{proof}

Finally, here is the example showing that in the absence of
self-indulgence, Theorem~\ref{T.main} fails.
In the development below, where we use square
roots of $7$ modulo powers of $3,$ we could, more generally, replace
$3$ by any prime $p,$ take any invertible irrational element $\alpha$
of the ring $\Z_p$ of $\!p\!$-adic integers, and look
at the images of $\alpha,-\alpha\in\Z_p$ in the rings $\Z/p^k\Z.$
The choice of a quadratic irrationality
just makes the presentation a little simpler.

\begin{example}\label{E.sqrt7}
A countable, downward directed family $F$ of subsets
of $\Z$ for which~\textup{(\ref{d.cupcap=0})} holds,
but~\textup{(\ref{d.capU=0})} does~not.
\end{example}

\begin{proof}[Construction and proof]
For each integer $k>0,$ let
\begin{equation}\begin{minipage}[c]{35pc}\label{d.sqt7}
$S(k)\ =\ \{x\in \Z\mid$ the image of $x$ in $\Z/3^k\Z$ is
either $0,$ or a square root of $7$ in that ring$\}.$
\end{minipage}\end{equation}

Since $S(1)\supseteq S(2)\supseteq\dots\,,$
the set $F=\{S(k)\}$ is downward directed.
To show that~(\ref{d.cupcap=0}) holds, let $g$ be any nonzero
member of $\Z,$ and $n$ any positive integer.
Choose a positive integer $k$ large enough so that
\begin{equation}\begin{minipage}[c]{35pc}\label{d.g2-7m2}
$3^k$ does not divide any of the $n+1$ nonzero
integers $g^2-7\,m^2$ with $0\leq m\leq n.$
\end{minipage}\end{equation}
(E.g., taken any $k$ such that $3^k>\max(g^2,\,7n^2).)$
Then I claim that $g\notin n S(k).$

Indeed, suppose we had
\begin{equation}\begin{minipage}[c]{35pc}\label{d.=g_inS(k)}
$g\ =\ g_0+\dots+g_{n-1},$ with all $g_i\in S(k).$
\end{minipage}\end{equation}
If we let $c$ denote a square root of $7$ in $\Z/3^k\,\Z,$
(which exists, by Hensel's Lemma \cite[Theorem~3.4.1]{p-adic}, and
is unique up to sign), then
by~(\ref{d.sqt7}), each of the $g_i$ in~(\ref{d.=g_inS(k)})
has residue modulo $3^k$ either $0,$ $c,$ or $-c.$
Hence~(\ref{d.=g_inS(k)}) implies that the residue of $g$
modulo $3^k$ has the form $m\,c$ for some integer $m$ of absolute
value $\leq n.$
Squaring, we conclude that $g^2 \equiv 7\,m^2\pmod{3^k},$
contradicting~(\ref{d.g2-7m2}).
This establishes~(\ref{d.cupcap=0}).

To show that~(\ref{d.capU=0}) fails, consider any
sequence $S(m_0),\,S(m_1),\,\dots$ of elements of $F,$ determined
by nonnegative integers $m_0,\,m_1,\dots\,.$
I claim that $U(S(m_0),S(m_1),\dots)=\Z;$ in fact, that
\begin{equation}\begin{minipage}[c]{35pc}\label{d.=Z}
$S(m_0)^*+S(m_1)^*+\,\dots\,+S(m_{3^{m_0}}^{\strt})^*\ =\ \Z.$
\end{minipage}\end{equation}

For let $c$ be a square root of $7$ in $\Z/3^{m_0}\,\Z.$
I claim that every $S(m_i)$ contains an integer $c_i$ whose
residue modulo $3^{m_0}$ is $c.$
For if $m_i \leq m_0,$ then
$S(m_i)$ contains $S(m_0),$ and so contains every element thereof
whose residue class modulo $3^{m_0}$ is $c,$ while if $m_i\geq m_0,$
then the residue class $c$ in $\Z/3^{m_0}\Z$ can be lifted
to a square root of $7$ in $\Z/3^{m_i}\Z$ (cf.\ proof of Hensel's
Lemma), a representative of which in $\Z$ will be the desired $c_i.$

Given any $g\in \Z,$ the element $g/c \in \Z/3^{m_0} \Z$ is the
residue of a nonnegative integer $h<3^{m_0}.$
Let us choose
elements $g_i\in S(m_i)$ for $i=1,\dots,3^{m_0}$ such that for exactly
$h$ values of $i,$ $g_i$ is the element $c_i$ chosen in the preceding
paragraph, while for the remaining values, $g_i = 0.$
Then the
sum $g_1+\dots+g_{3^{m_0}}^{\strt}$ is congruent
modulo $3^{m_0}$ to $h\,c,$
which by choice of $h$ is congruent to $g.$
On the other hand, $S(m_0)$ contains all multiples of $3^{m_0}$
(see~(\ref{d.sqt7})), so by choosing $g_0\in S(m_0)$
to be an appropriate one of these, we can get exact equality,
\begin{equation}\begin{minipage}[c]{35pc}\label{d.+g3=g}
$g\ =\ g_0+g_1+\dots+g_{3^{m_0}}^{\strt},$
\end{minipage}\end{equation}
as required to establish~(\ref{d.=Z}),
and hence falsify~(\ref{d.capU=0}).
\end{proof}

One can get similar examples by replacing the group of $\!3\!$-adic
integers implicit in the above construction with other examples
of a topological group $S$ containing
a subgroup $G$ and a cyclic dense subgroup $H$
having trivial intersection.
(In the above example, $S=\Z_3,$ $G=\Z$ and $H=\lang\,\sqrt{7}\ \rang.)$
For instance, one can take $S=\mathbb{R}/\Z,$
let $G$ be its dense subgroup $\mathbb{Q}/\Z,$ and
let $H$ be the subgroup generated by
the image $\beta$ of an irrational $b\in\mathbb{R}.$
Letting $F$
consist of the intersections of $G$ with a family of neighborhoods
of $\{-\beta,\,0,\,\beta\}\subseteq S$ under the usual topology,
one gets the same sort of behavior as in Example~\ref{E.sqrt7}.

\section{Remarks on self-indulgent sets}\label{S.remarks}

Though the concept of a self-indulgent set of subsets
of $G$ has proved useful, it is not clear that we have
formulated the best version of it.
Originally, I thought it would be enough to require that
for every family $(x_T)_{T\in F}$ there should exist a cofinal
subset $F'\subseteq F$ making $(x_T)_{T\in F'}$ converge strongly
to $x$:
I thought this would imply the condition now used, that for
every such family indexed by a cofinal subset $F'\subseteq F,$
one can get strong convergence on a smaller cofinal subset
$F''\subseteq F'.$
But I was unable to prove this.

Before settling on the present fix for that
problem, I considered other possibilities.
For instance, instead of looking at cofinal subsets of $F,$
one might look at isotone maps $f$ of arbitrary
downward directed posets $I$ into $F,$ having downward cofinal images.
Convergence of the system $(c_{f(i)})_{i\in I}$ with
respect to the ordering on $I$ would be a weaker condition
than convergence with respect to the ordering on the image set $f(I).$
Perhaps some variant of that idea could still be useful.

One may also ask whether examples can be found of self-indulgent
families essentially different from our co-$\!\kappa\!$ filters.
The answer is, ``Yes, but \dots''.
The lemma below gives such examples, but they require
knowing in advance the topology one is aiming at, so they are
of no evident use in getting new applications of Theorem~\ref{T.main}.
%

\begin{lemma}\label{L.compact}
Let $G$ be a locally compact Hausdorff topological abelian group,
and let $F$ be
the set of all compact neighborhoods of $0$ in $G$ \textup{(}or
any downward cofinal subset thereof\textup{)}.
Then $F$ is self-indulgent.
\end{lemma}

\begin{proof}
Because $G$ is locally compact, $F$ is a neighborhood basis
of $0$ in $G,$ so strong convergence with respect to $F$ is
equivalent to convergence.

Now for all $S\in F,$ compactness of $S^*$ implies that every
system of points indexed by a directed set has a cofinal
subsystem which converges to a point of $S^*;$
so in particular, we have the cases of this condition
required by the definition of self-indulgence.
\end{proof}

One may ask whether for $F$ a self-indulgent family that yields
a Hausdorff topology on a group $G,$ the members of $F$ must
become compact under that topology.
The difficulty, when one tries to prove this, is that
the self-indulgence condition only applies to families of
points indexed by cofinal subsets of $F,$ while compactness would
require a like condition for families indexed
by arbitrary directed sets.
In a similar vein, I.\,V.\,Protasov (personal communication)
has asked whether
under a topology so induced, the group $G$ must be complete.

\section{The nonabelian case}\label{S.nonab}

In this section we drop the assumption
that our group is abelian, letting
$G$ be an arbitrary group, written multiplicatively,
and see how the statement
and proof of Theorem~\ref{T.main} can be adapted to this situation.
In view of our multiplicative notation, we will denote
the identity element of $G$ by $e,$ and for $S\subseteq G$ write
\begin{equation}\begin{minipage}[c]{35pc}\label{d.S^*mult}
$S^*\ =\ S\cup\{e\}\cup S^{-1}.$
\end{minipage}\end{equation}

In \cite[\S3.1-\S3.2]{P+Z} Protasov and Zelenyuk likewise
generalize their results to noncommutative groups.
(Cf.\ also \cite[\S1.3]{YZ_bk}.)
As the analog of the sums $\sum_{i<n}\,S^*_i$
of~(\ref{d.U}), they use the union, over all permutations of the
index set $n,$ of the corresponding permuted product of the $S^*_i$
(and then, as in~(\ref{d.U}), take the union of this over all $n).$

We will take a different approach here.
Let us first note that it will not work to simply
replicate the definition~(\ref{d.U}) with sums
$S^*_0+S^*_1+\dots+S^*_{n-1}$
replaced by products $S^*_0\,S^*_1\,\dots\,S^*_{n-1}.$
The trouble is that we
cannot say that $\bigcup_{n\in\omega}\,S^*_0\,S^*_1\,\dots\,S^*_{n-1}$
will contain the product of two sets of the same sort; essentially
because $\omega$ does not contain a union of two
successive copies of itself as an ordered set.

So let us use an index set which does.
Let
\begin{equation}\begin{minipage}[c]{35pc}\label{d.Q}
$Q\ =$ a totally ordered set of the order-type of the rational numbers.
\end{minipage}\end{equation}
(We do not call this $\mathbb{Q}$ because we are not interested
in its algebraic structure, but only in its order-type.
In fact, for our one explicit calculation,
in the proof of Lemma~\ref{L.S_in_U},
a different realization of this order-type will prove convenient.)

Given any $\!Q\!$-tuple $(S_q)_{q\in Q}$ of subsets of $G,$ let
\begin{equation}\begin{minipage}[c]{35pc}\label{d.UQ}
$U((S_q)_{q\in Q})\ =\ \bigcup_{q_0<\dots<q_n\in Q}
\,S^*_{q_0}\dots\,S^*_{q_n},$
\end{minipage}\end{equation}
where the union is over all finite increasing sequences in $Q.$
The sets~(\ref{d.UQ}) have the property
which we just noted that $\!\omega\!$-indexed products lack;
indeed, it is easy to see

\begin{lemma}\label{L.UU}
Let $(S_q)_{q\in Q}$ be a family of subsets of $G,$
and let $\sigma,\,\tau:Q\to Q$ be two order-embeddings
such that $\sigma(q)<\tau(q')$ for all $q,\,q'\in Q.$
Then
\begin{equation}\begin{minipage}[c]{35pc}\label{d.UU}
$U((S_{\sigma(q)})_{q\in Q})\ U((S_{\tau(q)})_{q\in Q})\ \subseteq
\ U((S_q)_{q\in Q}).$\qed
\end{minipage}\end{equation}
\end{lemma}

The next result shows that sets of the form $U((S_q)_{q\in Q})$
can be made small enough to do what we will need.

\begin{lemma}\label{L.S_in_U}
If $\T$ is a group topology on $G,$ and $S$ a neighborhood
of $e$ under $\T,$ then one can choose for each $q\in Q$ a neighborhood
$S_q$ of $e$ under $\T$ so that $U((S_q)_{q\in Q})\subseteq S.$
\end{lemma}

\begin{proof}
(Cf.\ \cite[proof of Lemma~3.1.1]{P+Z}.)
Let $T_0=S,$ and choose recursively for each $i>0$ a neighborhood
$T_i$ of $e$ in $\T$ so that $T_i\,T_i\,T_i\subseteq T_{i-1}.$
Identify $Q$ as an ordered set with the set of
those rational numbers in the
unit interval $(0,1)$ of the form $m/2^i,$ and
\begin{equation}\begin{minipage}[c]{35pc}\label{d.m/}
for each $q=m/2^i,$ written in lowest terms, let $S_q=T_i.$
\end{minipage}\end{equation}
Then I claim that $U((S_q)_{q\in Q})\subseteq S.$

To show this, it suffices to show that for all finite sequences
$q_0<\dots<q_n\in Q$ we have $S_{q_0}^*\dots\,S_{q_n}^*\subseteq S.$
If we take a common denominator $2^j$ for all members of such a
finite sequence, then by enlarging the sequence we can assume
without loss of generality that
$\{q_0,\dots,q_n\}$ is the whole set
\begin{equation}\begin{minipage}[c]{35pc}\label{d.m/2j}
$\{m/2^j\ \mid\ 0<m<2^j\}.$
\end{minipage}\end{equation}
Let us now enlarge the finite product of sets $S_q$ determined
by~(\ref{d.m/2j}) still further, by changing those factors
whose index $q$ has the largest possible denominator,
$2^j,$ from $S_q=T_j$ to the larger set $T_{j-1}.$
(This will help in an induction to come.)

If we now classify the elements of~(\ref{d.m/2j})
into those which, expressed in lowest terms,
have denominator $2^j,$ those having denominator $2^{j-1},$ and those
with smaller denominators,
we see that each term with denominator $2^{j-1}$ is flanked
on each side by terms with denominator $2^j,$ and that
the resulting $\!3\!$-term strings of indices with denominators
$2^j,$ $2^{j-1},$ $2^j$ are disjoint.
In the modified product of subsets of $G$ that we have
described, the factors corresponding to
these strings of three terms have the form $T_{j-1}\,T_{j-1}\,T_{j-1}.$
By assumption, this product is contained in $T_{j-2}.$
Replacing each product $T_{j-1}\,T_{j-1}\,T_{j-1}$ with the possibly
larger set $T_{j-2},$ we conclude that our product of subsets is
contained in a product of the same form, but with subscripts
now running not over~(\ref{d.m/2j}) but over the
elements of $Q$ with denominator $\leq 2^{j-1}.$
(Note that ``of the same form'' includes the condition
that elements $q$ with largest possible denominator, now
$2^{j-1},$ are assigned the set $T_{j-2}$ rather than $T_{j-1}.)$

Iterating this reduction, we conclude that our product is contained in
one with the single index element $1/2^1,$ which is
assigned the set $T_{1-1}=T_0=S,$ giving the desired inclusion.
\end{proof}

(Amusing observation: The set $Q$ used in the above
proof has a natural order-isomorphism with the set
of intervals deleted in the ``middle third'' construction of the Cantor
set (arranged from left to right); and if we think of the relation
$T_i\,T_i\,T_i\subseteq T_{i-1}$ in the above proof intuitively
as saying that $T_i$ has one-third the ``weight'' of $T_{i-1},$ then the
weights of these sets can be taken to agree with the
lengths of those deleted intervals.
Thus, the above proof is related to the fact
that the total length of those deleted intervals is $1.)$

In studying the finest group topology in which a given downward
directed set $F$ converges to $e,$ it will be convenient to
require that $F$ be closed under conjugation by elements
of $G;$ i.e., that for every $S\in F$ and $g\in G$ we have
$gSg^{-1}\in F.$
The following lemma allows us to reduce the general case to that case.

\begin{lemma}\label{L.F^G}
Let $F$ be a downward directed family of nonempty subsets of $G,$
and \textup{(}following \cite[Definition~3.1.6]{P+Z}\textup{)}
let us write $F^G$ for the set of all
subsets of $G$ of the form $\bigcup_{g\in G}\,g\,S_g\,g^{-1},$
for $\!G\!$-tuples $(S_g)_{g\in G}$ of members of $F.$

Then $F^G$ is again a downward directed family of nonempty subsets
of $G,$ it is invariant under conjugation by elements of $G,$ and for
every group topology $\T$ on $G,$ the family $F^G$ converges
to $e$ under $\T$ if and only if $F$ does.
\end{lemma}

\begin{proof}
That $F^G$ is downward directed follows from the fact
that $F$ is, and it is conjugation invariant by construction.
From the fact that each set $\bigcup_{g\in G}\,g\,S_g\,g^{-1}\in F^G$
contains a member of $F,$ namely $S_e,$ it follows that if
$F^G$ converges to $e$ under $\T$ (i.e., if it has members contained
in every $\!\T\!$-neighborhood of $e),$ then so does $F.$

Now suppose, conversely, that $F$ converges to $e$ under $\T,$
and let $S$ be any neighborhood of $e$ in $\T.$
For each $g\in G,$ the set $g^{-1}S\,g$ is also
a neighborhood of $e,$ hence contains some $S_g\in F,$
which is to say that $S$ contains $g S_g g^{-1}.$
Thus $S$ will contain $\bigcup_{g\in G}\,g\,S_g\,g^{-1}\in F^G;$
so $F^G$ also converges to $e,$ as required.
\end{proof}

Restricting attention to conjugation-invariant families $F,$
we can now give the analog of~(\ref{d.U-basis}).

\begin{proposition}[{cf.\ \cite[Theorem~3.1.4]{P+Z}, \cite[Theorem~1.17]{YZ_bk}}]\label{P.U-basis}
Let $F$ be a downward directed family of nonempty subsets of $G,$
which is closed under conjugation by members of $G.$
Then the sets $U((S_q)_{q\in Q})$
defined by~\textup{~(\ref{d.UQ})}, where
$(S_q)_{q\in Q}$ ranges over all $\!Q\!$-tuples of members of $F,$
form a basis of open neighborhoods of $e$ in a group
topology $\T_F $ on $G,$ which is
the finest group topology under which $F$ converges to $e.$
\end{proposition}

\begin{proof}
It is easy to see that the family of
sets $U((S_q)_{q\in Q})$ is downward directed,
is closed under conjugation by elements of $G$ (because $F$ is),
is closed under inverses (since for each $q\in Q,$ $(S^*_q)^{-1}=S^*_q,$
hence if we let $\sigma:Q\to Q$ be an order-antiautomorphism, we get
$U((S_q)_{q\in Q})^{-1}=U((S_{\sigma(q)})_{q\in Q})),$
and has the property that each member of the family contains a
product of two other members (by Lemma~\ref{L.UU}).

To conclude that these sets give a basis of open neighborhoods of $e$
in a group topology on $G,$ it remains to show that for every
such set $U((S_q)_{q\in Q})$ and element
$x\in U((S_q)_{q\in Q}),$ there exists another such set
$U((T_q)_{q\in Q})$ with
\begin{equation}\begin{minipage}[c]{35pc}\label{d.xUT}
$x\,U((T_q)_{q\in Q})\ \subseteq\ U((S_q)_{q\in Q}).$
\end{minipage}\end{equation}
To see that this holds, note that by~(\ref{d.UQ}), $x$ lies in a
finite product $S^*_{q_0}\dots\,S^*_{q_n}$ with
$q_0<\dots<q_n\in Q.$
Now $\{q\in Q\mid q_n<q\}$ is an order-isomorphic
copy of $Q;$ let us write it $\tau(Q)$ where $\tau:Q\to Q$
is an isotone map.
Thus, letting $T_q=S_{\tau(q)},$ we get~(\ref{d.xUT}).

So our sets give a basis of open sets for a group topology $\T_F.$
Moreover, $F$ converges to $e$ in this topology, since
each $U((S_q)_{q\in Q})$ contains members of $F;$
indeed, contains each of the $S_q.$

To show that $\T_F$ is the finest group topology on $G$ under
which $F$ converges to $e,$ suppose $\T$ is any such topology.
For every open neighborhood $S$ of $e$ in $\T,$
Lemma~\ref{L.S_in_U} gives us a set
of the form $U((S'_q)_{q\in Q})$ contained in $S,$
with each $S'_q$ an open neighborhood of $e$ under $\T.$
By the assumption that $F$ converges to $e$ under $\T,$ each $S'_q$
contains some $S_q\in F,$ hence
$U((S_q)_{q\in Q})\subseteq U((S'_q)_{q\in Q})\subseteq S$
is a neighborhood of $e$ under $\T_F$
contained in $S;$ so $\T_F$ is at least as fine as $\T.$
\end{proof}

We have thus generalized to nonabelian groups $G$ the concepts and
results on abelian $G$ quoted in~\S\ref{S.intro}
as~(\ref{d.abelian})--(\ref{d.U-basis}).
The definitions and results of our earlier development immediately
following these
(the remaining material in \S\S\ref{S.intro}--\ref{S.peculiar})
go over to the nonabelian case with minimal change.
Indeed, the argument that gave us Corollary~\ref{C.cupcap=0}, applied
to Proposition~\ref{P.U-basis}, gives

\begin{corollary}\label{C.cupcap=0_nonab}
If $F$ is a conjugation-invariant downward directed
family of subsets of $G,$ then a
{\em necessary} condition for the topology $\T_F$ to be Hausdorff is
\begin{equation}\begin{minipage}[c]{35pc}\label{d.cupcap=e}
$\bigcup_{n>0}\,\bigcap_{S\in F_{\strt}}\,(S^*)^n\ =\ \{e\}.$
In other words, for every $g\in G-\{e\}$ and every $n>0,$
there exists $S\in F$ with $g\notin (S^*)^n.$\qed
\end{minipage}\end{equation}
\end{corollary}

The analogs of Definitions~\ref{D.strong} and~\ref{D.self-indulgent}
are

\begin{definition}\label{D.strong,_SI_nonab}
If $F$ is a downward directed family of subsets of $G,$
and $(x_i)_{i\in I}$ a family of elements of $G$ indexed by a
downward directed partially ordered set $I,$
we shall say that $(x_i)$ \emph{converges strongly} to
an element $x\in G$ with respect to $F$ if for every $S\in F,$
there exists $i\in I$ such that for all $j\leq i,$ $x_j\,x^{-1}\in S^*.$

A downward directed family $F$ of subsets of $G$
will be called \emph{self-indulgent} if for
every $S\in F,$ and every family $(x_T)_{T\in F'}$ of elements
of $S^*$ indexed by a downward cofinal
subset $F'\subseteq F,$ there exist an $x\in S^*$
and a downward cofinal subset $F''\subseteq F'$ such that
$(x_T)_{T\in F''}$ converges strongly to $x$ with respect to $F.$
\end{definition}

(The above definition of strong convergence is not right-left symmetric,
since it uses $x_j\,x^{-1}$ rather than $x^{-1}x_j.$
However, since the family of elements $(x_j\,x^{-1})_{j\in I}$
is conjugate, by $x,$ to $(x^{-1}\,x_j)_{j\in I},$
one can deduce that if $F$ is closed under conjugation by members
of $G,$ the condition becomes symmetric.)

The proof that co-$\!\kappa\!$
filters are self-indulgent also goes over with no change.
We state this below, along with another fact,
immediate to verify, that we will need.

\begin{lemma}\label{L.Fr-self-ind_nonab}
Let $X$ be any infinite subset of $G,$ and $\kappa$ any regular
cardinal $\leq\card(X).$
Then the co-$\!\kappa\!$ filter $F$ on $X$ is self-indulgent.

Moreover, if $X$ is invariant under conjugation by
elements of $G,$ then that filter $F$ is likewise closed under
conjugation by elements of $G.$\qed
\end{lemma}

We now come to the analogs of the material of \S\ref{S.main}.
A little care is needed in generalizing Lemma~\ref{L.one-more},
though the ideas are the same.

\begin{lemma}\label{L.one-more_nonab}
Let $F$ be a self-indulgent downward directed system of subsets of
$G$ which is closed under conjugation by members of $G,$
and satisfies~\textup{(\ref{d.cupcap=e})}.
Let $g\in G,$ and suppose that for some $n\geq 0$
and $0\leq m\leq n,$ $S_0,\dots,S_{m-1},\,S_{m+1},\dots,S_n$
are members of $F$ such that
\begin{equation}\begin{minipage}[c]{35pc}\label{d.n-1_nonab}
$g\notin S^*_0 \dots S^*_{m-1}\,S^*_{m+1} \dots S^*_n.$
\end{minipage}\end{equation}
Then there exists $S_m\in F$ such that
\begin{equation}\begin{minipage}[c]{35pc}\label{d.n_nonab}
$g\notin S^*_0 \dots S^*_{m-1}\,S^*_m\ S^*_{m+1} \dots S^*_n.$
\end{minipage}\end{equation}
\end{lemma}

\begin{proof}
As before, the contrary assumption says that
for each $T\in F,$ we can choose $n+1$ elements
\begin{equation}\begin{minipage}[c]{35pc}\label{d.nth_nonab}
$g_{0,T}\in S_0^*,\ \ \dots,\ \ g_{m{-}1,T}\in S^*_{m-1},
\ \ g_{m,T}\in T,\ \ g_{m{+}1,T}\in S^*_{m+1},\ \ \dots,
\ \ g_{n,T}\ \in\ S_n^*$
\end{minipage}\end{equation}
(note how the $\!m\!$-th condition differs from the others), such that
\begin{equation}\begin{minipage}[c]{35pc}\label{d.=g_nonab}
$g_{0,T}\,\dots\,g_{m{-}1,T}\ \,g_{m,T}\ \,g_{m{+}1,T}\,\dots
\,g_{n,T}\ =\ g.$
\end{minipage}\end{equation}
(However, in writing expressions like the
above, we will, from this point on, omit the terms indexed by $m-1$ and
$m+1,$ and only show those indexed by $0,$ $m$ and $n.)$

Making $n$ successive applications of our self-indulgence
assumption on $F$ (we did these from left to right in
proving Lemma~\ref{L.one-more}; but the order makes no difference),
we can get elements
\begin{equation}\begin{minipage}[c]{35pc}\label{d.hi}
$g_i$ for $0\leq i\leq n,$ where for $i\neq m,$ $g_i\in S_i^*,$
while $g_m=e,$
\end{minipage}\end{equation}
and a cofinal subfamily $F'\subseteq F,$
such that for each $i,$ the family $(g_{i,T})_{T\in F'}$
converges strongly to $g_i$ with respect to $F.$
Defining
\begin{equation}\begin{minipage}[c]{35pc}\label{d.g'i}
$g'_{i,T}\ =\ g_{i,T}\,g_i^{-1},$
\end{minipage}\end{equation}
we conclude that
\begin{equation}\begin{minipage}[c]{35pc}\label{d.g'->e}
for each $i\in\{0,\dots,n\},$
the family of elements $(g'_{i,T})_{T\in F'}$ converges
strongly to $e$ with respect to $F.$
\end{minipage}\end{equation}

Now~(\ref{d.n-1_nonab}) and the condition in~(\ref{d.hi})
imply that $g\neq g_0\dots g_m\dots g_n,$ so let us write
\begin{equation}\begin{minipage}[c]{35pc}\label{d.g'neqe}
$g'\ =\ g\cdot(g_0\,\dots\,g_m\,\dots\,g_n)^{-1}\ \neq\ e.$
\end{minipage}\end{equation}
Since $F$ satisfies~(\ref{d.cupcap=e}), we can find $S\in F$ such that
\begin{equation}\begin{minipage}[c]{35pc}\label{d.g'notin_nonab}
$g'\ \notin\ (S^*)^{n+1}.$
\end{minipage}\end{equation}

On the other hand, note that if in the right-hand side
of~(\ref{d.g'neqe}) we expand the initial factor $g$
using~(\ref{d.=g_nonab}), and then use~(\ref{d.g'i})
to rewrite each of the resulting factors $g_{i,T}$
as $g'_{i,T}\,g_i,$ we get
\begin{equation}\begin{minipage}[c]{35pc}\label{d.new=g_nonab}
$g'\ =\ (g'_{0,T}\,g_0)\,\dots\,(g'_{m,T}\,g_m)\,\dots\,
(g'_{n,T}\,g_n)\,
(g_0\,\dots\,g_m\,\dots\,g_n)^{-1}$ for all $T\in F'.$
\end{minipage}\end{equation}
Letting $h_i=g_0\dots g_{i-1}$ for $0\leq i\leq n,$ this becomes
\begin{equation}\begin{minipage}[c]{35pc}\label{d.cj}
$g'\ =\ (h_0\,g'_{0,T}\,h_0^{-1})\,\dots\,
(h_m\,g'_{m,T}\,h_m^{-1})\,\dots\,
(h_n\,g'_{n,T}\,h_n^{-1})$ for all $T\in F'.$
\end{minipage}\end{equation}
From the facts that the $g'_{i,T}$ all converge strongly
to $e$ with respect to $F,$ and that $F$ is closed
under conjugation by members of $G,$ it follows that
in~(\ref{d.cj}), each of the factors $h_i\,g'_{i,T}\,h_i^{-1}$
converges strongly to $e.$
Hence for some $T\in F',$ all the factors
of~(\ref{d.cj}) lie in the $S^*$ of~(\ref{d.g'notin_nonab}).
That instance of~(\ref{d.cj}) therefore
contradicts~(\ref{d.g'notin_nonab}), completing the proof of the lemma.
\end{proof}

Given $F$ as in the above lemma, and any $g\in G-\{e\},$ we can use
that lemma to build up, by recursion with respect to
any enumeration of $Q$ by the natural numbers,
systems $(S_q)_{q\in Q}$ such that $g\notin U(S_q)_{q\in Q}.$
We deduce

\begin{theorem}\label{T.main_nonab}
Let $F$ be a downward directed
system of subsets of $G$ which is self-indulgent, and
closed under conjugation by all elements of $G.$
\textup{(}In particular,
by Lemma~\ref{L.Fr-self-ind_nonab} this is true if for
some conjugation-invariant $X\subseteq G$ and some $\kappa\leq\card(X),$
$F$ is the co-$\!\kappa\!$ filter on $X.)$
Then the finest group topology on $G$ under which $F$ converges to $e$
is Hausdorff if and only if $F$
satisfies~\textup{(\ref{d.cupcap=e})}.\qed
\end{theorem}

It is not clear to me how closely related this is to
the nearest result in \cite{P+Z}, Theorem~3.2.1.
That result is restricted to countable groups $G,$ but concerns
the finest group topology under which a general sequence (equivalently,
the cofinite (i.e., co-$\!\aleph_0)\!$ filter on a general subset,
{\em not} necessarily conjugation invariant) converges.
The criterion given for that topology to be Hausdorff uses,
in place of the $\!n\!$-fold products
implicit in~(\ref{d.cupcap=e}), arbitrary group words
$f(x_0,\dots,x_n)$ in $n+1$ variables, and
constants from $G,$ which satisfy $f(e,\dots,e)=e.$
These two sorts of expressions ultimately reduce to the same thing;
but the quantification of the conditions is subtly different.
Perhaps this is not surprising:~(\ref{d.capU=0})
and~(\ref{d.cupcap=0}) can also be looked at as similar
conditions which involve different quantifications, but which become
equivalent in the case of self-indulgent~$F.$

In \cite[\S\S3.3,~3.4]{P+Z}, topologies on {\em rings} determined
by families of subsets are similarly studied.

\section{A Fibonacci connection}\label{S.Fib}

Many interesting applications are given in \cite{P+Z} of the criterion
obtained there for the cofinite filter on a
countable subset of an
abelian group to converge to $0$ in a Hausdorff group topology.
In particular, it is shown that there exist such topologies on $\Z$
under which various integer sequences -- for instance the
Fibonacci sequence \cite[Corollary~2.2.8]{P+Z} --
converge to $0.$

Note that in the nonabelian free group $G=\lang\,x,y\,\rang$
on two generators, one can define a Fibonacci-like sequence by
\begin{equation}\begin{minipage}[c]{35pc}\label{d.Fib}
$f_{0}=x,\quad f_{1}=y,\quad f_{n+1}=f_{n-1}f_n\quad (n\in\Z).$
\end{minipage}\end{equation}
I had hopes of proving that there was
a Hausdorff group topology on $\lang\,x,y\,\rang$ under which this
sequence converged to $e.$
However, if we define an automorphism $\varphi$
of $\lang\,x,y\,\rang$ by $\varphi(x)=y,$ $\varphi(y)=xy,$
then we see that in~(\ref{d.Fib}), $f_n=\varphi^n(x);$ so the result I
hoped for would imply that every $g\in \lang\,x,y\,\rang$ satisfied
$\lim_{n\to\infty}\varphi^n(g)=e.$
But calculation shows that the commutator $x\,y\,x^{-1} y^{-1}$ is
fixed by $\varphi^2;$ so this cannot be true.
Indeed, there cannot even exist a Hausdorff group topology under
which the sequence $f_n$ approaches
some fixed element $c$ of $G,$ or of a topological overgroup of $G,$
since then we would have
\begin{equation}\begin{minipage}[c]{35pc}\label{x.[]=1}
$\varphi^n(x\,y\,x^{-1} y^{-1})=
\varphi^n(x)\,\varphi^{n+1}(x)\,\varphi^n(x)^{-1} \varphi^{n+1}(x)^{-1}
\to\,c\,c\,c^{-1} c^{-1}=e,$
\end{minipage}\end{equation}
though as noted, the left-hand side has, for every
even $n,$ the value $x\,y\,x^{-1} y^{-1}.$
However, I don't see any obstruction to there being a topological
overgroup of $G$ under which the values of $f_{2n}$
and $f_{2n+1}$ each approach constant values.

For another context in which the ``Fibonacci automorphism'' $\varphi$
of $\lang\,x,y\,\rang$ (there called $\sigma^{1/2})$
comes up, see~\cite{submonoid}.

\section{Final remark, and acknowledgements}\label{S.thanks}

I do not know of interesting applications of the results of this note.
My motivation has been structural:
``What ideas underlie the arguments of~\cite{P+Z};
and in what more general contexts are those ideas applicable?''
Perhaps group theorists will find such applications.

I am indebted to Dikran Dikranjan, Pace Nielsen, Igor V.\,Protasov,
K.\,M.\,Rangaswamy and Yevhen Zelenyuk
for helpful comments and corrections to previous versions of this
note, and for references to the literature.


\begin{thebibliography}{00}

\bibitem{submonoid} George M.\,Bergman,
{\em On monoids, 2-firs and semifirs,}
preprint, 2013, 28\,pp., readable at
\url{http://math.berkeley.edu/~gbergman/papers}
and at \url{http://arxiv.org/abs/1309.0564}\,.

\bibitem{UA} P.\,M.\,Cohn,
{\em Universal algebra,} Second edition,
Mathematics and its Applications, 6.
D.\,Reidel Publishing Co., 1981. xv$+$412 pp.
MR~{\bf 82j}:08001.

\bibitem{E+M} Paul C.\,Eklof and Alan H.\,Mekler, 
{\em Almost free modules.  Set-theoretic methods,}
Revised edition. North-Holland Mathematical Library, 65, 2002.
MR~{\bf 2003e}:20002 (1st ed. {\bf 92e}:20001).

\bibitem{p-adic} Fernando Q. Gouv\^{e}a,
{\em $p$-adic numbers, an introduction,}
Springer, Universitext, 1993. vi$+$282 pp.
MR~{\bf 95b}:11111.
2nd ed. 1997, vi$+$298 pp.
MR~{\bf 98h}:11155.

\bibitem{PJH} P.\,J.\,Higgins,
{\em Introduction to topological groups,}
London Mathematical Society Lecture Note Series, No. 15.
Cambridge University Press, London-New York, 1974. v$+109$pp.
MR~{\bf 50}\#13355.

\bibitem{P+Z} I.\,Protasov and E.\,Zelenyuk, 
{\em Topologies on groups determined by sequences,}
Mathematical Studies Monograph Series, 4.
VNTL Publishers, L'viv, 1999.
111\,pp.
MR~{\bf 2001i}:22005.

\bibitem{YZ_bk} Yevhen G. Zelenyuk,
{\em Ultrafilters and topologies on groups},
de Gruyter Expositions in Mathematics, v.\,50, 2011. viii$+219$pp.
MR~{bf 2012c}:22002.

\end{thebibliography}
\end{document}